\documentclass[sigconf]{acmart}
\renewcommand\footnotetextcopyrightpermission[1]{} 

\AtBeginDocument{%
  \providecommand\BibTeX{{%
    \normalfont B\kern-0.5em{\scshape i\kern-0.25em b}\kern-0.8em\TeX}}}

\usepackage{amssymb}
\usepackage{tabularx}
\usepackage{amssymb}
\usepackage{graphicx}
\usepackage{epstopdf}
\usepackage{enumerate}
\usepackage{amsthm} 

\usepackage[bookmarks=false]{hyperref}

\begin{document}
\title{On Universal Scaling of Distributed Queues under Load Balancing} 

\author{Xin Liu}
\email{xliu272@asu.edu}
\affiliation{
  \institution{School of Electrical, Computer and Energy Engineering,  Arizona State University, Tempe, AZ}
}

\author{Lei Ying}
\email{leiying@umich.edu}
\affiliation{
  \institution{Electrical Engineering and Computer Science Department\\
  The University of Michigan, Ann Arbor, MI}
}

\begin{abstract}
This paper considers the steady-state performance of load balancing algorithms in a many-server system with distributed queues. The system has $N$ servers, and each server maintains a local queue with buffer size $b-1,$ i.e. a server can hold at most one job in service and $b-1$ jobs in the queue. Jobs in the same queue are served according to the first-in-first-out (FIFO) order. The system is operated in a heavy-traffic regime such that the workload per server is $\lambda = 1 -  N^{-\alpha}$ for $0.5\leq \alpha<1.$ We identify a set of algorithms such that the steady-state queues have the following universal scaling, where {\em universal} means that it holds for any $\alpha\in[0.5,1)$: (i) the number of of busy servers is $\lambda N-o(1);$ and (ii) the number of servers with two jobs (one in service and one in queue) is $O(N^{\alpha}\log N);$ and (iii) the number of servers with more than two jobs is $O\left(\frac{1}{N^{r(1-\alpha)-1}}\right),$ where $r$ can be any positive integer independent of $N.$ The set of load balancing algorithms that satisfy the sufficient condition includes join-the-shortest-queue (JSQ), idle-one-first (I1F), and power-of-$d$-choices (Po$d$) with $d\geq N^\alpha\log^2 N.$ We further argue that the waiting time of such an algorithm is near optimal order-wise.  
\end{abstract}

\maketitle

\section{Introduction}
The rapid growth of cloud computing, online social networks, Internet-of-things (IoT), and big data analytics have brought unprecedented volume of data  to data centers \cite{Benson_10, Roy_15, Jeon_18} at an unprecedented speed. Load balancing, which balances the workload across servers in a data center to optimize resource utilization and minimize response times, is at the heart of modern data center operations \cite{Rajesh_13, Borg_15, Eisenbud_16}. For example, it has been reported in \cite{Zats_12, Farhan_15} that extra 100 ms in response time can  lead to 1\% loss in revenue for online retail platforms like Amazon.

While a large-scale data center is operated under a lightly-loaded condition (i.e., the workload is significant lower than its capacity) most of the time, load balancing becomes vital when the system is heavily loaded (i.e., the load approaches to the system capacity) because it is the occasional heavy-load that affects the user experience the most.    
This paper focuses on performance and fundamental limits of load balancing algorithms in large-scale data centers with distributed queues, where jobs are dispatched immediately to servers upon arrival and each server maintains a local queue. We consider a range of heavy-traffic regimes, parameterized by a single heavy-traffic parameter $\alpha.$ We assume the workload per server is $\lambda = 1 - N^{-\alpha}$ for $0.5\leq \alpha<1.$ We remark that the steady-state performance for $\alpha<0.5$ has been analyzed in a recent paper \cite{LiuYin_18}. 
However, little is known about the steady-state performance of load balancing for  $0.5 \leq \alpha < 1,$  which is the focus of this paper. We establish the following results, which complement the results in \cite{LiuYin_18} and provide a comprehensive characterization of the steady-state performance of distributed queues in heavy-traffic regimes. The main results include:
\begin{itemize}
    \item[(i)] We consider a set of load balancing algorithms (denoted by $\Pi$), which includes popular load balancing algorithms such as JSQ, I1F, and Po$d$ with $d\geq N^\alpha\log^2 N.$ Define $S_i$ to be the fraction of servers with at least $i$ jobs at steady state. For any algorithm in $\Pi,$  we establish an upper bound on the $r$th moment of the following metric $$\max\left\{\sum_{i=1}^{b} S_i-1 -\frac{k\log N}{N^{1-\alpha}},0\right\},$$ where $r$ is any positive constant independent of $N.$ The proof is based on Stein's method \cite{BraDaiFen_15,BraDai_17,Yin_16} for queueing systems and by proving State-Space Collapse (SSC) using Lyapunov drift method \cite{ErySri_12}. 
    
    \item[(ii)] Using the moment bounds, we show that under any algorithm in $\Pi,$ the waiting probability of a job and the mean waiting time  are both $O\left(\frac{\log N}{N^{1-\alpha}}\right).$ In other words, these load balancing algorithms achieve asymptotically zero waiting time (so zero delay) for $\alpha<1$ while maintaining full efficiency asymptotically ($\rho\rightarrow 1$ for any $\alpha>0$). 
    
    \item[(iii)] We further characterize the steady-state queue lengths under any algorithm in $\Pi.$ Specifically, the following universal scaling holds:
    \begin{itemize}
    \item The average number of busy servers is $\lambda N-o(1).$
    \item The average number of servers with two jobs, i.e., one in service and one in queue, is $O\left(N^{\alpha}\log N\right).$
    \item The average number of servers with more than two jobs is $o(1),$ i.e., it is rare to have a server with three or more jobs.
    \end{itemize}
    \end{itemize}

We remark the most interesting point of result (iii) is the number of servers with two jobs because the result says the number of servers with at least three jobs are rare and the conclusion that the number of busy servers is almost $\lambda N-o(1)$ can be easily seen from the work-conserving property. The result in \cite{LiuYin_18} shows that when $\alpha<0.5,$ the number of jobs waiting in the system is $o(1)$, i.e. rarely, there is any job waiting; and for $0.5\leq \alpha<1,$ the number of jobs waiting is close to $N^\alpha=\frac{1}{1-\lambda}.$ This  phase-transition occurring at $\alpha=0.5$ coincides with what has been observed in a many-server system with a single shared queue \cite{HalWhi_81}, the M/M/N system, where it has been shown that the waiting probability is zero when $\alpha<0.5,$ one when $\alpha>0.5,$ and  a nontrivial constant only when $\alpha=0.5.$ Because of this celebrated result \cite{HalWhi_81}, the heavy-traffic regime with $\alpha=0.5$ is called the Halfin-Whitt regime. While a many-server system with distributed queues behaves fundamentally differently from the M/M/N system, it is interesting that for both systems, the phase-transition occurs at $\alpha=0.5.$

In Figure \ref{fig:contribution}, we plot  a $\log$-scaled version of the number of servers with two jobs (i.e. $\log E[N(S_2-S_3)] / \log N$) versus the heavy-traffic parameter $\alpha$ for $0<\alpha<1$ under JSQ. We note that the steady-state result for $0<\alpha<0.5$ was recently established in \cite{LiuYin_18}, which corresponds to pink line in the figure.  \cite{Bra_18} proved that $E[NS_2]$ is $O(\sqrt{N})$ in the Haffin-Whitt regime (i.e. $\alpha = 0.5$) at steady state, which is red dot in the figure.  \cite{GuptaWalton_17} proved a diffusion result for $\alpha=1$ that the scaled ($1/N$) total queue length is a constant, which implies $E[NS_2]=O(N)$ and it corresponds to green dot in the figure. We, however, note that the result in \cite{GuptaWalton_17} is proved for the process level only, and it remains open whether the same result holds for the steady-state. This paper completes the understanding of steady-state queue lengths for a large set of load balancing algorithms, and provides the following universal scaling result: 
$$E[NS_1]=\lambda N-o(1), \quad E[NS_2]=O(N^\alpha\log N), \quad E[NS_3]=o(1).$$ We conjecture these bounds above differ from the tight bounds by a logarithmic factor. Specifically, we {\em conjecture} that the following results hold for any load balancing algorithm in the set $\Pi:$
$$E[N S_1]=\lambda N-o(1), \quad E[N S_2]=\Theta(N^\alpha), \quad E[N S_3]=o(1);$$ and these bounds are asymptotically optimal. 

\begin{figure}[htbp]
  \centering
  \includegraphics[width=2.6in]{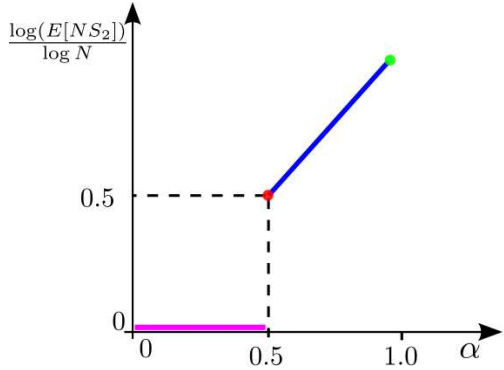}
  \caption{Our contributions and related work.}
  \label{fig:contribution}
\end{figure}

Figure \ref{fig:scale} is a simulation result that shows $\frac{E[NS_2]}{N^\alpha}$ under JSQ for $\alpha = 0.5,$ $0.75,$ and $1.0.$ According to our results and conjecture,  $N^{1-\alpha}E[S_2]$ should remain almost as a constant even as $N$ increases. This can be seen clearly from the figure, which further confirms our main results.  
\begin{figure}[!htbp]
  \centering
  \includegraphics[width=2.8in]{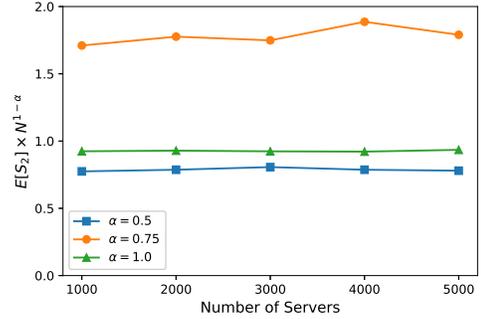}
  \caption{Scaling of $N^{1-\alpha} E[S_2]$ versus $N$ under JSQ}
  \label{fig:scale}
\end{figure}

\subsection{Related Work}
Heavy-traffic analysis of many server systems, motivated by call centers, was pioneered  by Shlomo Halfin and Ward Whitt  in their seminal work \cite{HalWhi_81}, where they considered the M/M/N system (also called the Erlang-C systemf) with load  $ 1 -\gamma N^{-\alpha}.$ They discovered that a phase-transition occurs when $\alpha=0.5,$ such that the waiting probability is asymptotically zero when $\alpha<0.5,$ is one when $\alpha>0.5,$ and is a nontrival value only when $\alpha=0.5.$ The heavy traffic regime with $\alpha = 0.5$, therefore, is called the Halfin-Whitt regime or the quality-and-efficiency-driven (QED) regime. \cite{HalWhi_81} assumes exponential service times.  For general service times, the distribution of queue lengths in the Halfin-Whitt regime  has been a topic of great interest since then (see \cite{Ree_09,PuhRee_10,GamGol_13} and references within).  Heavy-traffic regimes other than the Halfin-Whitt regime have also been considered. For example, \cite{Atar_12} considered the non-degenerate slowdown regime, where  $\alpha=1$ and the mean waiting time is comparable with the mean service time. They proved that the total queue length with a proper scaling ($1/N$) converges to a diffusion process, and the result has been generalized in \cite{He_15} to general service time distributions. In a recent paper, \cite{BraDaiFen_15} provided a universal characterization (for any $0\leq \alpha \leq 1$) of the steady-state performance of the Erlang-C system  based on Stein's method. Similarly, this paper, together with \cite{LiuYin_18}, provide a universal characterization of many-server systems with distributed queues.

Recent studies on many-server systems with distributed queues are motivated by the proliferation of cloud computing and large-scale data centers, where computing tasks (such as search, data mining, etc) are routed to a server upon arrival, instead of waiting at a centralized queue. Analyzing the performance of many-server queueing system is known to be difficult, in particular, when the number of servers $N$ is large.  A significant result in this area is recent work \cite{EscGam_18}, which proved that the system converges to a two-dimensional diffusion process under the join-the-shortest-queue algorithm (JSQ) \cite{Win_77,Web_78} in the Halfin-Whitt regime. Specifically, at the process-level (i.e. over a finite time interval), it proved the scaled process ($1/\sqrt{N}$) counting the number of idle servers and servers with exactly two jobs weakly converges to a two-dimensional reflected Ornstein-Uhlenbeck (OU) process. \cite{Bra_18} proved the diffusion limit of JSQ is valid at steady-state. \cite{BanerMukherjee2019} provided further characterization of the steady distribution of the diffusion process established in \cite{EscGam_18}.
Based on the diffusion limits in \cite{EscGam_18}, \cite{MukBorvan_18} showed that the same diffusion limit is valid under the power-of-$d$ choices with a properly chosen $d.$
\cite{GuptaWalton_17} considers the non-degenerate slowdown ($\alpha=1$), and establishes the diffusion limit of the total queue length scaled by $1/N.$ By using the diffusion limit, it compares the delay performance of various load balancing algorithms, including JSQ, JIQ and I1F. We note the diffusion limit was established over any finite time interval. It remains open whether the same results in \cite{GuptaWalton_17} hold at steady-state. Furthermore, the system has constant queueing delays when $\alpha=1,$ so the queueing behaviors and the analysis in \cite{GuptaWalton_17} are fundamentally different from this paper. 
The work mostly related to ours is \cite{LiuYin_18}, which considers the steady-state performance of a set of load balancing algorithms, including JSQ, power-of-$d$-choices (Po$d$) \cite{Mit_96,VveDobKar_96}, idle-one-first (I1F) \cite{GuptaWalton_17}, join-the-idle-queue (JIQ) \cite{LuXieKli_11,Sto_15} in the sub-Halfin-Whitt regime (i.e. for $0 < \alpha < 0.5$). The results show that the number of servers with at least two jobs is negligible $o(1),$ which establish the pink line in Figure \ref{fig:contribution}. Our paper complements the results in \cite{LiuYin_18} and establishes the blue line in Figure \ref{fig:contribution}, a universal scaling for $0.5\leq \alpha<1.$
  
\section{Model and Main Results}
We consider a large-scale system with $N$ homogeneous servers. We assume job arrival is a Poisson process with rate $\lambda N$ and service times follow an exponential distribution with rate one. As shown in Figure \ref{model-pod}, each server maintains a separate queue
and buffer size is $b-1$ (i.e., each server can have one job in service and $b-1$ jobs in queue). Jobs in a queue are served in an First-in-First-out (FIFO) order.
We focus on the traffic regime such that $\lambda=1-N^{-\alpha}$ with $0.5\leq\alpha<1.0.$ 

\begin{figure}[!htbp]
  \centering
  \includegraphics[width=3.0in]{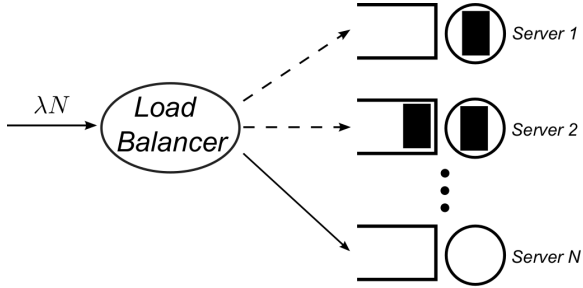}
  \caption{Load Balancing in Many-Server Systems.}
  \label{model-pod}
\end{figure}

Denote by $S_i(t)$ to the fraction of servers with at least $i$ jobs at time $t \geq 0.$ Under the finite buffer assumption with buffer size $b-1$, we define $S_i(t) = 0,\  \forall i \geq b+1,\ \forall t\geq 0,$ for notational convenience. Furthermore, define set ${\mathbb S}\subseteq {\mathbb R}^b$ such that $$\mathbb S = \{ s\in{\mathbb R}^b ~|~ 1\geq s_1\geq \cdots \geq s_{b}\geq  0\hbox{ and } ~Ns_{i}\in \mathbb N,\ \forall i\}.$$ We have $S(t) = [S_1(t), S_2(t), \cdots ,S_b(t)]^T \in \mathbb S$ for any $t\geq 0.$ 
We consider load balancing algorithms which route each incoming job to a server upon its arrival based on $S(t)$ so that ($S(t):t\geq 0$) is a finite-state and irreducible continuous-time Markov chain (CTMC), which implies that ($S(t):t\geq 0$) has a unique stationary distribution.

Let $S \in \mathbb S$ be the random variables having the stationary distribution of ($S(t):t\geq 0$). Note $\lambda,$ ($S(t):t\geq 0$), and $S$ all depend on $N,$ the number of servers in the system. {Let $A_i(s)$ denote the probability that an incoming job is routed to a server with at least $i$ jobs when the system is in state $s \in \mathbb S,$ e.g.  $$A_1(s)= \Pr\left(\left. \text{an incoming job is routed to a busy server} \right| {S(t)=s}\right).$$ } 
Define a set of load balancing $\Pi$ to be
\begin{align*}
    \Pi = \left\{ \pi ~|~ \text{under load balancing~} \pi, A_1(s)\leq \frac{1}{\sqrt{N}}, s_1\leq 1-\frac{1}{4N^\alpha}, \right. \\ 
    \left. A_2(s)\leq 10\left(\frac{2r}{N^{1-\alpha}}\right)^r,  s_2\leq 0.95, \text{~and~} A_b(s) \leq s_b, \forall s \in{\mathbb S} \right\}. 
\end{align*}

A load balancing algorithm in $\Pi$ implies that (i) for any given state $s$ in which at least $\frac{1}{4N^\alpha}$ fraction of servers are idle, an incoming job should be routed to an idle server with probability at least $1-\frac{1}{\sqrt{N}};$
(ii) for any given state $s$ in which at least 5\% servers have two job or less, the probability an incoming job is routed to a server with at least two jobs should be no more than $10\left(\frac{2r}{N^{1-\alpha}}\right)^r;$ and (iii) given state $s,$ the probability that a job is dropped because of being routed to a server with full buffer is upper bounded by that under a random routing algorithm, which is $s_b.$
There are several well-known algorithms that satisfy this condition.
\begin{itemize}
\item {\bf Join-the-Shortest-Queue (JSQ)}: JSQ routes an incoming job to the least loaded server in the system, so $A_1(s)=0$ for $s_1\leq 1-\frac{1}{4N^\alpha};$ $A_2(s) =  1_{\{s_{2}=1, s_{3}<1\}} = 0,$ for $s_2 < 0.95;$ and $A_b(s) \leq s_b.$

\item {\bf Idle-One-First (I1F)}: I1F routes an incoming job to an idle server if available and else to a server with one job if available. Otherwise, the job is routed to a randomly selected server. Therefore, $A_1(s)=0$ for $s_1\leq 1-\frac{1}{4N^\alpha};$ $A_2(s) =  1_{\{s_{2}=1, s_{3}<1\}} = 0,$ for $s_2 < 0.95;$ and $A_b(s) \leq s_b.$

\item {\bf Power-of-$d$-Choices (Po$d$)}:  Po$d$ samples $d$ servers uniformly at random and dispatches the job to the least loaded server among the $d$ servers. Ties are broken uniformly at random. Given $d \geq N^\alpha\log^2 N,$ $A_1(s)\leq \frac{1}{\sqrt{N}}$ for $s_1\leq 1-\frac{1}{4N^\alpha};$ $A_2(s) = s_2^d \leq (0.95)^d \leq 10 \left(\frac{2r}{N^{1-\alpha}}\right)^r;$ and $A_b(s) \leq s_b.$

\end{itemize}

We first have the following moment bounds which are instrumental for establishing the main results of this paper. The proof of this theorem is presented in Section \ref{sec:proof}. 
\begin{theorem} \label{Thm:main}
Assume $\lambda=1- N^{-\alpha}$ for $0.5\leq\alpha<1$ and buffer size $b.$ For any  load balancing algorithms in $\Pi,$ the following bound holds at steady-state for $N$ and positive integer $r$ such that $\frac{N^{1-\alpha}}{32\log N} > r,$
\begin{align*}
E\left[\left(\max\left\{\sum_{i=1}^{b} S_i-1 -\frac{k\log N}{N^{1-\alpha}},0\right\}\right)^r\right]\leq 10 \left(\frac{2r}{N^{1-\alpha}}\right)^r,
\end{align*} where $k = 32r b+1.$ \qed
\end{theorem}

{Note the expectation in Theorem \ref{Thm:main} is with respect to the stationary distribution of the CTMC ($S(t):t\geq 0$) according to the definition of $S.$} Based on Theorem \ref{Thm:main}, we have the  {\em universal} scaling results in Corollary \ref{Thm:S3} and asymptotic zero waiting results in Corollary \ref{Thm:zerodelay}.  

To establish the universal scaling results, in Corollary \ref{Thm:S3}, we first show that almost no server has more than two jobs under a load balancing algorithm in $\Pi.$    
\begin{corollary} \label{Thm:S3}
Assume $\lambda=1-N^{-\alpha}$ for $0.5\leq\alpha<1.$ Under any load balancing algorithm in ${\Pi},$ the following results hold for any $N$ such that $\frac{N^{1-\alpha}}{k\log N} \geq 5,$ 
$$E[S_3] \leq 20 \left(\frac{3r}{N^{1-\alpha}}\right)^r.$$
\end{corollary}.  

Next, we analyze the waiting time, waiting probability for algorithms in $\Pi,$ and the steady-state queues. 
Let $\mathcal W$ denote the event that an incoming job is routed to a busy server, and $p_{\mathcal W}$ denote the probability of this event at steady-state. Let $\mathcal B$ denote the event that an incoming job is blocked (discarded) and $p_{\mathcal B}$ denote the probability of this event at steady-state. {Note the  ${\mathcal B}\subseteq {\mathcal W}$ because an incoming job is blocked when being routed to a busy server with $b$ jobs.}
Furthermore, let $W$ denote the steady-state waiting time of those jobs that are not blocked.

\begin{corollary} \label{Thm:zerodelay}
Assume $\lambda=1-N^{-\alpha}$ for $0.5\leq\alpha<1.$ Under load balancing algorithm in $\Pi$ and assume any positive constant $r$ such that
$N^{1-\alpha} \geq 3(40)^{\frac{r}{2}}r,$ we have $$E\left[W\right] \leq \frac{4k\log N}{N^{1-\alpha}},$$ and $$ p_{\mathcal W}\leq {20 \left(\frac{3r}{N^{1-\alpha}}\right)^{\frac{r}{2}}} +\frac{2k\log N}{N^{1-\alpha}}.$$ We furthermore have 
    $$\lambda N - 10N\left(\frac{3r}{N^{1-\alpha}}\right)^{\frac{r}{2}} \leq E\left[NS_1\right]\leq \lambda N,$$ so 
      $$E\left[NS_1\right]=\lambda N -o(1);$$ 
    and 
    $$E\left[NS_2\right]\leq {10N\left(\frac{3r}{N^{1-\alpha}}\right)^{\frac{r}{2}}}+{2kN^\alpha \log N}=O(N^\alpha \log N).$$

\end{corollary}

In the M/M/N system where a centralized queue is maintained for complete resource pooling, the average waiting time per job is $O\left(\frac{1}{N^{1-\alpha}}\right).$ In load balancing systems, Corollary \ref{Thm:zerodelay} suggests the waiting time to be $O\left(\frac{k\log N}{N^{1-\alpha}}\right).$  Therefore, the expected waiting of a load balancing algorithms in $\Pi$ is close to that in the M/M/N system when $N$ is large. It implies load balancing algorithms in $Pi$ have near optimal delay performance since the mean waiting time of the M/M/N system is a lower bound on that of any many-server systems with distributed queues. We conjecture that the average waiting time of load balancing algorithms in $\Pi$ is $\Theta\left(\frac{1}{N^{1-\alpha}}\right)$ as in the M/M/N system. The additional term  $k\log N,$ however, is needed in establishing a state-space-collapse result due to technical reasons. 

\section{State-space-collapse in the heavy traffic regime and Simple fluid model}
To prove the main theorem and corollaries, we first need to understand the system dynamic under the load balancing algorithms in $\Pi.$
In particular, we focus on those states where the total queue length is larger than $1+\frac{k\log N}{N^{1-\alpha}},$ so the truncated distance function used in Theorem \ref{Thm:main} has a non-zero value. In this region of the state space, we have a key observation that the system collapses to a smaller region, i.e. state-space-collapse (SSC) occurs, which is critical for establishing the main theorem.
We next explain the intuition behind SSC and present the formal statement in Lemma \ref{ssc:tail}.

For ease of exposition, we consider JSQ in a simple two-dimensional system with buffer size $b=1$, so we only need to consider $s_1$ and $s_2$. Given system state $(s_1, s_2)$ such that $s_1 < 1,$ in a fluid sense, $s_1$ increases with rate $\lambda-s_1\approx 1-s_1$ and $s_2$ decreases with rate $s_2,$ because all arrivals are routed to idle servers under JSQ when $s_1 < 1.$ Therefore, $s_1$ increases quickly and approaches to one  (the green region in the figure). Since most servers are busy  ($s_1$ is close to $1$) in the green region, the total queue length per server (i.e $s_1+s_2$) decreases with rate  $\lambda-(s_1+s_2)\approx \lambda-1,$ i.e. decreasing slowly. Therefore, starting from any state outside of the green and the shallow pink regions, the system will move quickly into the green region and then slowly towards the $s_1$-axis. Due to the difference of the scales (speeds) of the dynamics, the stochastic system at the steady-state is either in the green region or in the shallow pink region with a high probability, as shown in  Lemma \ref{ssc:tail}.
\begin{figure}[!htbp]
  \centering
  \includegraphics[width=2.7in]{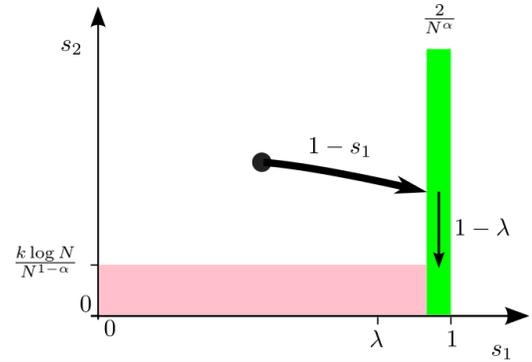}
  \caption{State Space Collapse}
  \label{fig:SSC}
\end{figure}

\begin{lemma}\label{ssc:tail}
For any load balancing in $\Pi,$ we have the following tail probability: For any $N$ such that $\frac{N^{1-\alpha}}{32\log N} > r$ and $k-\frac{r}{N^\alpha \log N} \leq \bar k \leq k,$ 
$$\Pr\left(\min\left\{ \sum_{i=2}^b S_i - \frac{\bar k\log N}{N^{1-\alpha}}, 1-S_1 \right\} \geq \frac{1}{2N^{\alpha}}\right)\leq \frac{1}{N^{2r}}.$$ \qed
\end{lemma}

Lemma \ref{ssc:tail} states that for any load balancing in $\Pi$, at steady state, with a high probability, either $S_1$ is larger than $1-\frac{1}{2N^\alpha}$ (most of servers are busy) or $\sum_{i=1}^b S_i$ is less than $\frac{\bar k \log N}{N^{1-\alpha}}$ (the number of jobs waiting is small). This is reasonable to expect  for a good load balancing $\pi$ in $\Pi$ (e.g. JSQ) because $s_i, \forall i\geq 2$ should not build up if idle servers exist. Lemma \ref{ssc:tail} is based on the  geometric-type bound in \cite{BerGamTsi_01} with Lyapunov function $V(s) = \min\left\{ \sum_{i=2}^b s_i - \frac{\bar k\log N}{N^{1-\alpha}}, 1-s_1 \right\}.$ The details can be found in the Appendix \ref{app:ssc}.

Now combine SSC in Lemma \ref{ssc:tail}, and the truncated distance function $$\max\left\{\sum_{i=1}^{b} s_i-1 -\frac{\bar k\log N}{N^{1-\alpha}},0\right\},$$
which is non-zero when $\sum_{i=2}^{b} s_i-\frac{\bar k\log N}{N^{1-\alpha}} > 1 - s_1.$ Therefore, with a high probability, the distance function is non-zero only if it is inside of the green region in Figure \ref{ssc:tail}. 

When the system is in the green region, in which most servers are busy ($s_1$ is close to $1$), we approximate it with a simple fluid system, whose arrival rate is $\lambda$ and departure rate is $1,$ i.e.,
$$\dot{x}=\lambda-1 =-\frac{1}{N^\alpha}\quad \hbox{when}\quad x>0.$$ 
One would expect that in the green region, the load balancing system behaves ``close" to the simple system, which implies the steady-state performances in terms of the truncated function of the two systems are also ``close". To formalize this intuition, in the following section, we will use Stein's method to couple the two systems with respect to the truncated distance function of total queue length in Theorem \ref{Thm:main}.

\section{Proof of Theorem \ref{Thm:main}} \label{sec:proof}

The steady-state function we consider in Theorem \ref{Thm:main} is a truncated distance function of the total queue length ($r$ is a positive integer)
\begin{align}
\left(\max\left\{\sum_{i=1}^b S_i-1-\frac{\bar k \log N}{N^{1-\alpha}},0\right\}\right)^r, \label{metric}  
\end{align} which measures the value by which the total queue length ($N\sum_{i=1}^b S_i$) exceeds $N+\bar{k}N^\alpha\log N$ at steady state. The function can be used to bound the waiting  probability and waiting time in Corollary \ref{Thm:zerodelay}, as well as to show `the result in Corollary \ref{Thm:S3}. To analyze \eqref{metric} in the original system, we utilize Stein's method, where we first solve Stein's equation  (also called the Poisson equation)  for the truncated function  \eqref{metric} in the simple system  and then bound the expected value of the truncated function using generator coupling.

\subsection{Stein's equation}
Recall the approximated simple system has arrival rate $\lambda$ and departure rate $1,$ i.e. 
\begin{align}
\dot{x}=-\frac{1}{N^\alpha}, \label{gen:L}
\end{align} where $\dot x=\frac{d x}{dt}.$ Define the distance function $$h_{\bar k}(x)=\max\left\{x-1-\frac{\bar k \log N}{N^{1-\alpha}},0\right\},$$ where $k-\frac{r}{N^\alpha\log N} \leq \bar k \leq k.$ Consider function $g: \mathbb R \to \mathbb R,$ such that it satisfies the following equation 
\begin{align}
\frac{d g(x)}{d t} =g'(x)\dot{x}= h^r_{\bar k}(x), \ \forall x \geq 0, \label{Gen:L0}
\end{align} i.e. 
\begin{align}
g'(x) \left(-\frac{1}{N^\alpha}\right) = h^r_{\bar k}(x), \ \forall x \geq 0. \label{Gen:L}
\end{align} according to \eqref{gen:L}, which is called Stein's equation.  

Now consider $x=\sum_{i=1}^b s_i.$ To bound the $r$th $E\left[h^r_{\bar k}\left(\sum_{i=1}^b S_i\right)\right]$  on the right-hand side of \eqref{Gen:L}, we need to analyze the left-hand side $$E\left[g'\left(\sum_{i=1}^b S_i\right) \left(-\frac{1}{N^\alpha}\right)\right].$$ Stein's method enables us to quantify the term by generator coupling as follows.

\subsection{Generator coupling}
For the original system, let $G$ be the generator of CTMC ($S(t):t\geq 0$). Given function $g: \mathbb R \to \mathbb R$ and the system state $S(t) = s,$ we have
\begin{align}
G g\left(\sum_{i=1}^b s_i\right) =&\lambda N (1-A_b(s)) \left(g\left(\sum_{i=1}^b s_i+\frac{1}{N}\right) - g\left(\sum_{i=1}^b s_i\right)\right) \nonumber\\
&+  N s_1 \left(g\left(\sum_{i=1}^b s_i-\frac{1}{N}\right) - g\left(\sum_{i=1}^b s_i\right)\right), \label{Gen:G}
\end{align}  where
\begin{itemize}
    \item $\lambda N(1-A_b(s))$ is the rate that the total queue length increases by one;
    \item $Ns_1$ is the rate that the total queue length decreases by one.
\end{itemize}
A detailed derivation of \eqref{Gen:G} can be found in Appendix \ref{app:gen-diff}. 

For any bounded function $g,$
\begin{align}
E\left[ G g\left(\sum_{i=1}^b S_i\right) \right] = 0, \label{steady-cond}
\end{align}
because the fact that $S$ represents steady-state of the CTMC.
Taking the expected value on both sides of \eqref{Gen:L} and combining it with \eqref{steady-cond}, we obtain
\begin{align}
E\left[h^r_{\bar k}\left(\sum_{i=1}^b S_i\right)\right] =& E\left[g'\left(\sum_{i=1}^b S_i\right) \left(-\frac{1}{N^\alpha}\right)\right] \nonumber\\
=& E\left[g'\left(\sum_{i=1}^b S_i\right) \left(-\frac{1}{N^\alpha}\right)\right] - E\left[ G g\left(\sum_{i=1}^b S_i\right) \right]. \label{Gen:diff}
\end{align}
Note the expectations in above equation are taken with respect to the distribution of $S$. Therefore, we are able to study the generator difference in \eqref{Gen:G} to bound $E\left[h^r_{\bar k}(\sum_{i=1}^b S_i)\right].$ 
Define $\eta = 1 + \frac{{\bar k}\log N}{N^{1-\alpha}}$ to simplify notation.
Define $\mathcal T_1 = \{ x ~|~ \eta - \frac{1}{N} \leq x \leq \eta + \frac{1}{N}\}$ and $\mathcal T_2 = \{ x ~|~  x > \eta + \frac{1}{N}\}.$
The generator difference \eqref{Gen:diff} is given in the following lemma.
\begin{lemma}\label{lemma:gen-diff}
\begin{align}
&E\left[h^r_{\bar k}\left(\sum_{i=1}^b S_i\right)\right]\nonumber\\
= &E\left[g'\left(\sum_{i=1}^b S_i\right)\left(\lambda A_b(S)- 1+S_1\right)\mathbb{I}_{\sum_{i=1}^b S_i \in \mathcal T_2}\right] \label{G-expansion-SSC}\\
&+E\left[\left(g'\left(\sum_{i=1}^b S_i\right)\left(-\frac{1}{N^\alpha}\right)-\lambda(1-A_b(S))g'(\xi) + S_1g'(\tilde{\xi})\right)\mathbb{I}_{\sum_{i=1}^b S_i \in \mathcal T_1}\right] \label{G-expansion-Gradient-1}\\
&-E\left[\frac{1}{2N}\left(\lambda(1-A_b(S))g''(\zeta) + S_1g''(\tilde{\zeta})\right)\mathbb{I}_{\sum_{i=1}^b S_i \in \mathcal T_2}\right]. \label{G-expansion-Gradient-2}
\end{align}
Here $\xi,\zeta  \in \left(\sum_{i=1}^b S_i,\sum_{i=1}^b S_i+\frac{1}{N}\right)$ and $\tilde{\xi},\tilde{\zeta}\in \left(\sum_{i=1}^b S_i-\frac{1}{N},\sum_{i=1}^b S_i\right)$ are random variables whose values depend on $\sum_{i=1}^b S_i.$ \qed
\end{lemma}
The proof of Lemma \ref{lemma:gen-diff} is in Appendix \ref{app:gen}. The following analysis provides upper bounds on \eqref{G-expansion-SSC}, \eqref{G-expansion-Gradient-1} and \eqref{G-expansion-Gradient-2}. The terms \eqref{G-expansion-Gradient-1} and \eqref{G-expansion-Gradient-2} are bounded by studying the first order derivative $g'$ and the second order derivative $g''$ as shown in Lemma \ref{lemma:g'+g''}; and the term \eqref{G-expansion-SSC} is bounded by using state space collapse in Lemma \ref{ssc:tail}.

\begin{lemma}\label{lemma:g'+g''}
\begin{align*}
\eqref{G-expansion-Gradient-1} + \eqref{G-expansion-Gradient-2} \leq \frac{2^{r+1}}{N^{r-\alpha}} + \frac{r E\left[h^{r-1}_{\bar k}\left(\sum_{i=1}^b S_i+{\frac{1}{N}}\right)\right]}{N^{1-\alpha}}.
\end{align*} \qed
\end{lemma}
The detailed proof of the lemma above can be found in Appendix \ref{app:g'}. Next, we consider the term \eqref{G-expansion-SSC}, which is based on the SSC result in Lemma \ref{ssc:tail}. According to Lemma \ref{ssc:tail}, we consider term \eqref{G-expansion-SSC} in two regions: $\Omega$ and its complementary $\bar \Omega$, where $$\Omega = \left\{ s ~\bigg |~ \min\left\{ \sum_{i=2}^b s_i - \frac{{\bar k}\log N}{N^{1-\alpha}}, 1-s_1 \right\} \leq \frac{1}{2N^\alpha} \right\}.$$ 

Considering \eqref{G-expansion-SSC} when $S\in \Omega,$ 
\begin{align*}
E\left[g'\left(\sum_{i=1}^b S_i\right)\left(\lambda A_b(S)- 1+S_1\right)\mathbb{I}_{\sum_{i=1}^b S_i \in \mathcal T_2}\mathbb{I}_{S \in \Omega}\right]
\end{align*}
we observe that $1-s_1\leq \frac{1}{2N^\alpha}$ which guarantees $(\lambda A_b(s)- 1+s_1)$ is small because $\sum_{i=1}^b s_i \in \mathcal T_2.$ 

Considering \eqref{G-expansion-SSC} when $S\in \bar \Omega,$ 
\begin{align*}
E\left[g'\left(\sum_{i=1}^b S_i\right)\left(\lambda A_b(S)- 1+S_1\right)\mathbb{I}_{\sum_{i=1}^b S_i \in \mathcal T_2}\mathbb{I}_{S \in \bar \Omega}\right]
\end{align*}
we apply the tail bound in Lemma \ref{ssc:tail}.

Collectively, we have the following lemma.
\begin{lemma}\label{lemma:ssc-term}
\begin{align}
\eqref{G-expansion-SSC} \leq \frac{1}{2}E\left[h^r_{\bar k}\left(\sum_{i=1}^{b} S_i\right)\right] 
+ \frac{b^r}{N^{2r-\alpha}}. \nonumber
\end{align} \qed
\end{lemma}

Recall that from Lemma \ref{lemma:gen-diff}, we have
$$E\left[h^r_{\bar k}\left(\sum_{i=1}^{b} S_i\right)\right] = \eqref{G-expansion-SSC}+\eqref{G-expansion-Gradient-1} + \eqref{G-expansion-Gradient-2}.$$
Applying Lemma \ref{lemma:g'+g''} and \ref{lemma:ssc-term} yields an upper bound on $E\left[h^r_{\bar k}\left(\sum_{i=1}^{b} S_i\right)\right]$ in terms of $E\left[h^{r-1}_{\bar k}\left(\sum_{i=1}^b S_i+\frac{1}{N}\right)\right],$ as stated in the following lemma.  

\begin{lemma}[Iterative Moment Bounds]\label{iter-r-th moment}
Assume $\lambda=1- N^{-\alpha},$ $0.5 \leq \alpha< 1.$ The following bound holds at steady-state for any positive integer $r$ such that:
\begin{align*}
E\left[h^{r}_{\bar k}\left(\sum_{i=1}^b S_i\right)\right]\leq \frac{2^{r+2}}{N^{r-\alpha}} + \frac{2r}{N^{1-\alpha}}E\left[h^{r-1}_{\bar k}\left(\sum_{i=1}^b S_i+\frac{1}{N}\right)\right]. 
\end{align*} \qed
\end{lemma}
\begin{proof}
Given upper bound on \eqref{G-expansion-Gradient-1} + \eqref{G-expansion-Gradient-2} in Lemma \ref{lemma:g'+g''} and upper bound on  \eqref{G-expansion-SSC} in Lemma \ref{lemma:ssc-term}, we have
\begin{align}
E\left[h^r_{\bar k}\left(\sum_{i=1}^{b} S_i\right)\right]
\leq & \frac{1}{2} E\left[h^r_{\bar k}\left(\sum_{i=1}^{b} S_i\right)\right] 
 +\frac{b^r}{N^{2r-\alpha}}\nonumber\\
&+\frac{2^{r+1}}{N^{r-\alpha}} + \frac{r}{N^{1-\alpha}}E\left[h^{r-1}_{\bar k}\left(\sum_{i=1}^b S_i + \frac{1}{N}\right)\right] \nonumber\\
\leq & \frac{1}{2} E\left[h^r_{\bar k}\left(\sum_{i=1}^{b} S_i\right)\right]  +\frac{2^{r+1}+1}{N^{r-\alpha}} \nonumber\\
& + \frac{r}{N^{1-\alpha}}E\left[h^{r-1}_{\bar k}\left(\sum_{i=1}^b S_i+ \frac{1}{N}\right)\right]\label{move}
\end{align} 
where the second inequality holds because
$$\frac{b^r}{N^{2r-\alpha}} \leq  \frac{1}{N^{r-\alpha}},$$ because $b < N.$

By moving the term $\frac{1}{2}E\left[h^{r}_{\bar k}\left(\sum_{i=1}^b S_i\right)\right]$ in \eqref{move} to the left-hand side, we have 
\begin{align*}
E\left[h^{r}_{\bar k}\left(\sum_{i=1}^b S_i\right)\right]\leq \frac{2^{r+2}+2}{N^{r-\alpha}} + \frac{2r}{N^{1-\alpha}}E\left[h^{r-1}_{\bar k}\left(\sum_{i=1}^b S_i+ \frac{1}{N}\right)\right].
\end{align*}
\end{proof}

\subsection{Proving Theorem \ref{Thm:main}}
We iteratively use Lemma \ref{iter-r-th moment} to establish Theorem \ref{Thm:main}. Define $w_r = \frac{2r}{N^{1-\alpha}}$ and $z_r = \frac{2^{r+2}+2}{N^{r-\alpha}},$ the inequality in Lemma \ref{iter-r-th moment} can be written as 
\begin{align*}
E\left[h^{r}_{\bar k}\left(\sum_{i=1}^b S_i\right)\right]\leq w_r E\left[h^{r-1}_{\bar k}\left(\sum_{i=1}^b S_i+ \frac{1}{N}\right)\right] + z_r,
\end{align*} where $k - \frac{r}{N^\alpha \log N} \leq \bar k \leq k.$ Continuing expanding the term $E\left[h^{r-1}_{\bar k}\left(\sum_{i=1}^{b} S_i+\frac{1}{N}\right)\right]$ yields  
\begin{align*}
E\left[h^{r}_k\left(\sum_{i=1}^{b} S_i\right)\right] &\leq \prod_{j=1}^r w_j  + \sum_{i=1}^{r-1} z_i \prod_{j=i+1}^r w_j + z_r\\
&\leq_{(a)} \prod_{j=1}^r w_j  + r z_1 \prod_{j=2}^r w_j\\
&\leq_{(b)} (r+1) z_1 \prod_{j=2}^r w_j \\
&\leq_{(c)} (r+1) z_1(w_r)^{r-1}\\
&\leq 10 \left(\frac{2r}{N^{1-\alpha}}\right)^r
\end{align*}
where $(a)$ holds because $z_i \prod_{j=i+1}^r w_j$ is decreasing for $1 \leq i \leq r-1$ because 
\begin{align*}
\frac{z_i \prod_{j=i+1}^r w_j}{z_{i-1} \prod_{j=i}^r w_j} = \frac{z_{i}}{z_{i-1} w_{i}} 
= \frac{\frac{2^{i+2}+2}{N^{i-\alpha}}}{\frac{2^{i+1}+2}{N^{i-1-\alpha}}\frac{2i}{N^{1-\alpha}}} \leq \frac{1}{2iN^\alpha} \leq 1;
\end{align*}
$(b)$ holds because $w_1=\frac{2}{N^{1-\alpha}} \leq z_1=\frac{10}{N^{1-\alpha}}$ implies $$\prod_{j=1}^r w_j  \leq z_1 \prod_{j=2}^r w_j;$$ and $(c)$ holds because $w_r$ is increasing in $r.$

\section{Proof of Corollary \ref{Thm:S3}}

To prove Corollary \ref{Thm:S3}, we first have the following lemma on $E[S_3].$
\begin{lemma} \label{lemma:S3}
At steady-state, we have   
$$E[S_3] = \lambda E[A_2(S) - A_b(S)].$$ \qed
\end{lemma}

The proof of Lemma \ref{lemma:S3} is in Appendix \ref{app:S3}. The key step is to choose a proper test function $f(s)=\sum_{i=3}^b s_i,$ and use the steady-state equation $E[Gf(S)]=0$.

From the lemma above, we have 
\begin{align}
E[S_3] \leq &  E\left[A_2(S)\right] \nonumber\\
       =&  E\left[A_2(S) | S_2 \geq 0.95\right]\Pr(S_2 \geq 0.95) \nonumber\\
       &+ E\left[A_2(S) | S_2 < 0.95\right]\Pr(S_2 < 0.95) \nonumber\\
       \leq& \Pr(S_2 \geq 0.95)+E\left[A_2(S) | S_2 < 0.95\right]. \label{S3_bounds}
\end{align}

The probability in \eqref{S3_bounds} can be upper bounded as follows: 
\begin{align*}
\Pr(S_2 \geq 0.95)
\leq&\Pr(S_1 + S_2 \geq 1.9)\\
\leq& \Pr\left(h_k\left(\sum_{i=1}^bS_i\right)\geq 0.9-\frac{k\log N}{N^{1-\alpha}}\right)\\
=& \Pr\left(h^r_k\left(\sum_{i=1}^bS_i\right)\geq \left(0.9-\frac{k\log N}{N^{1-\alpha}}\right)^r\right)\\
\leq& \frac{E\left [h^r_k\left(\sum_{i=1}^bS_i\right)\right ]}{\left(0.9-\frac{k\log N}{N^{1-\alpha}}\right)^r}\\
\leq& 10 \left(\frac{3r}{N^{1-\alpha}}\right)^r
\end{align*}
where the last inequality holds because $\frac{N^{1-\alpha}}{k\log N} \geq 5.$

The conditional expectation in \eqref{S3_bounds} is bounded
$$E\left[A_2(S) | S_2 < 0.95\right]\leq 10\left(\frac{2r}{N^{1-\alpha}}\right)^r$$
for any load balancing algorithms in $\Pi.$

\section{Proof of Corollary \ref{Thm:zerodelay}} \label{sec:0-delay}

We prove Corollary \ref{Thm:zerodelay} with the following steps: $(i)$ bound the blocking probability $p_{\mathcal B};$ $(ii)$ study the expected waiting time $E[W]$ based on $p_{\mathcal B};$ and $(iii)$ study the waiting probability $p_{\mathcal W}$ based on $p_{\mathcal B}$ and $E[W].$

Define $\delta_b = \sqrt{10} \left(\frac{3r}{N^{1-\alpha}}\right)^{\frac{r}{2}}.$ We next study the blocking probability $p_{\mathcal B}$ by considering two regions:
\begin{align*}
p_{\mathcal B}=&\Pr\left(\mathcal B\left|S_b\leq \delta_b\right.\right)\Pr\left(S_b\leq \delta_b\right)\\
&+\Pr\left(\mathcal B\left|S_b> \delta_b\right.\right)\Pr\left(S_b> \delta_b \right)\\
\leq& \Pr\left(\mathcal B\left|S_b\leq \delta_b\right.\right) + \Pr\left(S_b> \delta_b \right).
\end{align*}
For load balancing in $\Pi$, we have
\begin{align*}
p_{\mathcal B} \leq& \delta_b + \Pr\left(S_b> \delta_b \right) \\
\leq& \delta_b + \Pr\left(S_3> \delta_b \right) \\
\leq& \delta_b + \frac{E[S_3]}{\delta_b} \\
\leq& 10 \left(\frac{3r}{N^{1-\alpha}}\right)^{\frac{r}{2}},
\end{align*}
{where the first inequality holds because $A_b(s) \leq s_b$ for any load balancing algorithm in $\Pi$}, the third inequality holds due to the Markov inequality, and the last inequality holds because of the upper bound on $E[S_3]$ established in Corollary \ref{Thm:S3}.

For jobs that are not discarded, the average queueing delay according to Little's law is
$$\frac{E\left[\sum_{i=1}^bS_i\right]}{\lambda(1-p_{\mathcal B})}.$$ Therefore, the average waiting time is
\begin{align*}
E[W]=&\frac{E\left[\sum_{i=1}^bS_i\right]}{\lambda(1-p_{\mathcal B})}-1 \\
\leq& \frac{1+\frac{k\log N}{N^{1-\alpha}}+\frac{20}{N^{1-\alpha}}}{\lambda(1-p_{\mathcal B})}-1\\
=&\frac{\frac{k\log N}{N^{1-\alpha}}+\frac{20}{N^{1-\alpha}}+\frac{1}{N^\alpha}+\lambda p_{\mathcal B}}{\lambda(1-p_{\mathcal B})}\\
\leq& \frac{3k\log N}{\lambda(1-p_{\mathcal B})}  \leq \frac{4k\log N}{N^{1-\alpha}}
\end{align*}
where the first inequality holds by letting $r=1$ in Theorem \ref{Thm:main}, and the last inequality holds due to the upper bound on $p_{\mathcal B}$ for a large $N$ such that $N^{1-\alpha} \geq 3 (40)^{\frac{2}{r}}r.$

From the work conservation law, we have 
$$E[S_1]=\lambda(1-p_{\mathcal B}),$$ which implies
$$\lambda - 10 \left(\frac{3r}{N^{1-\alpha}}\right)^{\frac{r}{2}} \leq E[S_1]\leq \lambda.$$ 
Now according to Theorem \ref{Thm:main}, we have $$E\left[\sum_{i=1}^b {S}_i\right]\leq 1 +\frac{k\log N}{N^{1-\alpha}}+\frac{20}{N^{1-\alpha}},~~ k = 32 rb,$$ which results in the upper bound  on $E[S_2]:$ $$E[S_2]\leq E\left[\sum_{i=2}^b S_i\right]\leq 10 \left(\frac{3r}{N^{1-\alpha}}\right)^{\frac{r}{2}} +\frac{(k+20)\log N}{N^{1-\alpha}}.$$

We now study the waiting probability $p_{\mathcal W}$. Define $\overline{\mathcal W}$ to be the event that a job is not blocked and $p_{\overline{\mathcal W}}$ to be the steady-state probability of $\overline{\mathcal W}$. {Applying Little's law to the jobs waiting in the buffer yields  $$\lambda p_{\overline{\mathcal W}} E[T_{Q}] = E\left[\sum_{i=2}^b S_i\right],$$ where $T_{Q}$ is the waiting time for the jobs waiting in the buffer.}
Since $E[T_{Q}]$ is lower bounded by one, we have
$$p_{\overline{\mathcal W}} \leq \frac{E\left[\sum_{i=2}^b S_i\right]}{\lambda}.$$
Finally, a job that is not routed to an idle server is either blocked or is routed to wait in a buffer, so 
\begin{align*}
p_{\mathcal W} =&  p_{\mathcal B}+p_{\overline{\mathcal W}} \\
\leq& p_{\mathcal B}+\frac{E\left[\sum_{i=2}^b S_i\right]}{\lambda} \\
\leq& 20 \left(\frac{3r}{N^{1-\alpha}}\right)^{\frac{r}{2}} +\frac{2k\log N}{N^{1-\alpha}}.
\end{align*}

\section{Conclusion}
In this paper, we established moment bounds for a set ($\Pi$) of load balancing algorithms in heavy-traffic regimes.
The set $\Pi$ includes JSQ, I1F and Po$d$ ($d \geq N^\alpha \log^2 N$). Under any algorithm in ${\Pi},$ the expected waiting time and the waiting probability  of an incoming job is asymptotically zero. We further established universal scaling properties of the steady-state queues under any algorithm in ${\Pi}.$

\bibliographystyle{ACM-Reference-Format}
\bibliography{inlab-refs}

\appendix
\section{Proof of $Gg(\sum_{i=1}^b s_i)$ in (\ref{Gen:G}).} \label{app:gen-diff}
Let $e_i$ to be a $b$-dimension vector with the $i$th entry to be $1/N$ and others are zeros. Recall $A_i(s)$ denotes the probability that an incoming job is routed to a server with at least $i$ jobs when the system is in state $s \in \mathbb S.$
Since $G$ is the generator of CTMC ($S(t):t\geq 0$), given function $f: \mathbb S \to \mathbb R,$ we have
\begin{align}
G f(s) = \sum_{i=1}^{b}&\lambda N (A_{i-1}(s)-A_i(s)) (f(s + e_i) - f(s)) \nonumber\\
&+  N (s_i - s_{i+1}) (f(s - e_i)-f(s)). \label{Gen:f}
\end{align}  

Now define
\begin{equation}
f(s)=g\left(\sum_{i=1}^b s_i\right). \label{eq:steinsolution}
\end{equation} Substituting $f(s)$ in \eqref{eq:steinsolution} into \eqref{Gen:f}, we obtain $G g(\sum_{i=1}^b s_i)$ such that
\begin{align}
G g\left(\sum_{i=1}^b s_i\right) =& \sum_{i=1}^{b}\lambda N (A_{i-1}(s)-A_i(s)) \left(g\left(\sum_{i=1}^b s_i + \frac{1}{N}\right) - g\left(\sum_{i=1}^b s_i\right)\right) \nonumber\\
&+  N (s_i - s_{i+1}) \left(g\left(\sum_{i=1}^b s_i - \frac{1}{N}\right)-g\left(\sum_{i=1}^b s_i\right)\right) \nonumber\\
=&  \lambda N (1-A_b(s)) \left(g\left(\sum_{i=1}^b s_i + \frac{1}{N}\right) - g\left(\sum_{i=1}^b s_i\right)\right) \nonumber\\
&+  N s_1 \left(g\left(\sum_{i=1}^b s_i - \frac{1}{N}\right)-g\left(\sum_{i=1}^b s_i\right)\right).\nonumber
\end{align}

\section{Proof of Lemma \ref{lemma:gen-diff}} \label{app:gen}

Given Equation \eqref{Gen:L}, we obtain its solution to be
\begin{equation}g(x)=-\frac{N^\alpha}{r+1}\left(x-1-\frac{\bar k\log N}{N^{1-\alpha}}\right)^{r+1} \mathbb{I}_{x\geq 1+\frac{\bar k\log N}{N^{1-\alpha}}},\label{eq:L}\end{equation} and
\begin{equation}g'(x)=-N^\alpha\left(x-1-\frac{\bar k\log N}{N^{1-\alpha}}\right)^{r} \mathbb{I}_{x\geq 1+\frac{\bar k\log N}{N^{1-\alpha}}}.\label{eq:Lder}\end{equation}

Recall $\eta = 1 +\frac{\bar k\log N}{N^{1-\alpha}},$ $\mathcal T_1 = \{ x ~|~ \eta - \frac{1}{N} \leq x \leq \eta + \frac{1}{N}\}$ and $\mathcal T_2 = \{ x ~|~  x > \eta + \frac{1}{N}\}.$
From the closed-forms of $g$ and $g'$ in \eqref{eq:L} and \eqref{eq:Lder}, we note that for any $x <  \eta,$
\begin{equation*}
g(x) = g'\left(x\right)=0.
\end{equation*} Also note that when $x>\eta+\frac{1}{N},$
\begin{equation}g'(x)=-N^\alpha\left(x-\eta\right)^{r},\end{equation} so for $x>\eta+\frac{1}{N},$
\begin{equation}g''(x)=-rN^\alpha\left(x-\eta\right)^{r-1}.\end{equation}

By using the mean-value theorem in region $[\eta-\frac{1}{N}, \eta+\frac{1}{N}]$ and the Taylor theorem in region $(\eta+\frac{1}{N}, \infty)$, we have
\begin{align}
g(x+\frac{1}{N})-g\left(x\right) 
=& \left(g(x+\frac{1}{N})-g\left(x\right)\right) \left(\mathbb{I}_{ x \in \mathcal T_1 } + \mathbb{I}_{ x \in \mathcal T_2 }\right) \nonumber\\
=&  \frac{g'(\xi)}{N}\mathbb{I}_{x \in \mathcal T_1} + \left(\frac{g'(x)}{N} +  \frac{g''(\zeta)}{2N^2}\right) \mathbb{I}_{ x \in \mathcal T_2 } \label{g+1/N}
\end{align}
and
\begin{align}
g(x-\frac{1}{N})-g\left(x\right) 
=& \left(g(x-\frac{1}{N})-g\left(x\right)\right) \left(\mathbb{I}_{ x \in \mathcal T_1 } + \mathbb{I}_{ x \in \mathcal T_2 }\right) \nonumber\\
=&  -\frac{g'(\tilde{\xi})}{N}1_{x \in \mathcal T_1} + \left(-\frac{g'(x)}{N} +  \frac{g''(\tilde{\zeta})}{2N^2}\right) \mathbb{I}_{ x \in \mathcal T_2 } \label{g-1/N}
\end{align}
where $\xi, \zeta \in (x,x+\frac{1}{N})$ and $\tilde{\xi}, \tilde{\zeta} \in (x-\frac{1}{N},x).$

Based on \eqref{g+1/N} and \eqref{g-1/N}, we have 
\begin{align*}
G g\left(x\right) =& \lambda (1-A_b(s)) \left( g'(\xi) \mathbb{I}_{x \in \mathcal T_1} + \left( g'(x) +  \frac{g''(\zeta)}{2N}\right) \mathbb{I}_{ x \in \mathcal T_2 }  \right) \\
&+s_1 \left( -g'(\xi)\mathbb{I}_{x \in \mathcal T_1} + \left(-g'(x) +  \frac{g''(\zeta)}{2N}\right) \mathbb{I}_{ x \in \mathcal T_2 } \right). 
\end{align*}

Considering the simple system in \eqref{Gen:L}, we have $g'(x) = 0, \forall x < \eta - \frac{1}{N}$ that 
\begin{align*}
g'\left(x\right)\left(-\frac{1}{N^\alpha}\right) = g'\left(x\right) \left(-\frac{1}{N^\alpha}\right) \left( \mathbb{I}_{ x \in \mathcal T_1} + \mathbb{I}_{ x \in \mathcal T_2} \right).
\end{align*}

From the results above, we have 
\begin{align*}
&E\left[h^r_{\bar k}\left(\sum_{i=1}^b S_i\right)\right]\nonumber\\
=&E\left[g'\left(\sum_{i=1}^b S_i\right)\left(-\frac{1}{N^\alpha}\right)\right]+E\left[G g\left(\sum_{i=1}^b S_i\right)\right]\nonumber\\
= &E\left[g'\left(\sum_{i=1}^b S_i\right)\left(\lambda A_b(S)- 1+S_1\right)\mathbb{I}_{\sum_{i=1}^b S_i \in \mathcal T_2}\right] \\
&+E\left[\left(g'\left(\sum_{i=1}^b S_i\right)\left(-\frac{1}{N^\alpha}\right)-\lambda(1-A_b(S))g'(\xi) + S_1g'(\tilde{\xi})\right)\mathbb{I}_{\sum_{i=1}^b S_i \in \mathcal T_1}\right] \\
&-E\left[\frac{1}{2N}\left(\lambda(1-A_b(S))g''(\zeta) + S_1g''(\tilde{\zeta})\right)\mathbb{I}_{\sum_{i=1}^b S_i \in \mathcal T_2}\right],
\end{align*}
where we have random variables $\xi,\zeta  \in \left(\sum_{i=1}^b S_i,\sum_{i=1}^b S_i+\frac{1}{N}\right)$ and $\tilde{\xi},\tilde{\zeta}\in \left(\sum_{i=1}^b S_i-\frac{1}{N},\sum_{i=1}^b S_i\right)$ whose values depend on $\sum_{i=1}^b S_i.$ 

\section{Proof of Lemma \ref{lemma:g'+g''}} \label{app:g'}

For any $x\in\left[ 1 +\frac{\bar k\log N}{N^{1-\alpha}}  -\frac{2}{N}, 1 +\frac{\bar k\log N}{N^{1-\alpha}}  +\frac{2}{N}\right],$ from \eqref{Gen:L},  we obtain
$$|g'(x)| \leq \frac{|x-1 -\frac{\bar k\log N}{N^{1-\alpha}}|^r}{\frac{1}{N^\alpha}}\leq \frac{\left(\frac{2}{N}\right)^r}{\frac{1}{N^\alpha}}=\frac{2^r}{N^{r-\alpha}},$$
which implies the term \eqref{G-expansion-Gradient-1} is bounded by
\begin{align*}
&E\left[\left(g'\left(\sum_{i=1}^b S_i\right)\left(-\frac{1}{N^{\alpha}}\right)-\lambda(1-A_b(S))g'(\xi) + S_1g'(\tilde{\xi})\right)\mathbb{I}_{\sum_{i=1}^b S_i \in \mathcal T_2} \right]\nonumber\\
\leq& \left(\lambda+\frac{1}{N^{\alpha}}+1\right)\frac{2^r}{N^{r-\alpha}} \leq \frac{2^{r+1}}{N^{r-\alpha}}. 
\end{align*} 

For $x> 1 +\frac{\bar k\log N}{N^{1-\alpha}},$ we have
$$g'(x)=\frac{\left(x-1 -\frac{\bar k\log N}{N^{1-\alpha}}\right)^r}{-\frac{1}{N^{\alpha}}},$$ which implies
\begin{eqnarray*}
g''(x)= \frac{r\left(x-1 -\frac{\bar k\log N}{N^{1-\alpha}}\right)^{r-1}}{-\frac{1}{N^{\alpha}}}.
\end{eqnarray*}
and
\begin{align*}
|g''(x)|&=\left|\frac{r\left(x-1 -\frac{\bar k\log N}{N^{1-\alpha}}\right)^{r-1}}{-\frac{1}{N^{\alpha}}}\right|\\
&= r N^{\alpha} \left(\max\left\{x-1 -\frac{\bar k\log N}{N^{1-\alpha}}, 0\right\}\right)^{r-1} \\
&= r N^{\alpha} h^{r-1}_{\bar k}\left(x\right). 
\end{align*}
Therefore, we have the following bound on \eqref{G-expansion-Gradient-2} 
\begin{align*}
-&E\left[\frac{1}{2N}\left(\lambda(1-A_b(S))g''(\zeta) + S_1g''(\tilde{\zeta})\right)\mathbb{I}_{ \sum_{i=1}^b S_i \in \mathcal T_2 }\right] \nonumber\\
\leq& E\left[\frac{1}{2N}\left(\lambda|g''(\zeta)| + S_1|g''(\tilde{\zeta})|\right)\mathbb{I}_{ \sum_{i=1}^b S_i \in \mathcal T_2 }\right] \nonumber \\
\leq& \frac{r E\left[h^{r-1}_{\bar  k}\left(\sum_{i=1}^b S_i+\frac{1}{N}\right)\right]}{N^{1-\alpha}}. 
\end{align*} 

\section{Proof of SSC in Lemma \ref{ssc:tail}} \label{app:ssc}
To prove Lemma \ref{ssc:tail}, we consider Lyapunov function
\begin{align}
V(s)=\min\left\{ \sum_{i=2}^b s_i - \frac{\bar k\log N}{N^{1-\alpha}}, 1-s_1 \right\}, \label{ly-func}
\end{align}
to study its drift in the following lemma.

\begin{lemma}\label{DriftBound}
Given any load balancing in $\Pi,$ we have 
\begin{align*}
\triangledown V(s)\leq \frac{2}{\sqrt{N}} -\frac{k}{b}\frac{\log N}{N^{1-\alpha}},
\end{align*}
for any state $s \in \mathbb S$ such that
\begin{align*}
V(s)\geq \frac{1}{4N^{\alpha}}.
\end{align*}
\end{lemma} \qed
\begin{proof}

Given $V(s) \geq \frac{1}{4N^{\alpha}},$ we have two cases.
\begin{itemize}
\item Case 1: 
$1-s_1 \geq V(s)= \sum_{i=2}^b s_i - \frac{\bar k\log N}{N^{1-\alpha}} \geq \frac{1}{4N^{\alpha}}.$ 

In this case, we have
\begin{align*}
\triangledown V(s) \leq& \lambda(A_1(s)- A_b(s))-s_2 \\
\leq& \frac{1}{\sqrt{N}} - s_2 \\
\leq& \frac{1}{\sqrt{N}} -\frac{1}{4bN^{\alpha}} - \frac{\bar k}{b}\frac{\log N}{N^{1-\alpha}} \\
\leq& \frac{1}{\sqrt{N}} - \frac{k}{b}\frac{\log N}{N^{1-\alpha}}
\end{align*} 
where the second inequality holds because we consider load balancing in $\Pi,$ and the third inequality holds because $s_2 \geq \frac{\sum_{i=2}^b s_i}{b} \geq  \frac{1}{4bN^{\alpha}} + \frac{\bar k}{b}\frac{\log N}{N^{1-\alpha}}$.

\item Case 2: 
$\sum_{i=2}^b s_i - \frac{\bar k\log N}{N^{1-\alpha}}\geq V(s)=1-s_1\geq \frac{1}{4N^{\alpha}}.$  
In this case, we have
\begin{align*}
\triangledown V(s)\leq&-\lambda(1-A_1(s))+(s_1-s_2)\\
=& s_1-s_2-\lambda+\lambda A_1(s)\\
\leq& \frac{3}{4N^{\alpha}}-s_2+\lambda A_1(s)\\
\leq& \frac{1}{\sqrt{N}} + \frac{3}{4N^{\alpha}} -\frac{1}{4bN^{\alpha}} - \frac{\bar k}{b}\frac{\log N}{N^{1-\alpha}} \\
\leq& \frac{2}{\sqrt{N}} - \frac{k}{b}\frac{\log N}{N^{1-\alpha}}
\end{align*}
\end{itemize}
where the second inequality holds because $s_1 \leq 1-\frac{1}{4N^\alpha},$ and the third inequality holds because we consider load balancing in $\Pi$ and $s_2 \geq \frac{\sum_{i=2}^b s_i}{b} \geq  \frac{1}{4bN^{\alpha}} + \frac{\bar k}{b}\frac{\log N}{N^{1-\alpha}}.$
\end{proof}

The drift analysis in Lemma \ref{DriftBound} says that once the system at state $s$ such that $s_1 \leq 1 - \frac{1}{4N^\alpha}$ and $\sum_{i=2}^b s_i \geq \frac{\bar k\log N}{N^{1-\alpha}}+\frac{1}{4N^\alpha},$ Lyapunov function has a negative drift in the order of $\frac{\log N}{N^{1-\alpha}}$, which is related to the choice of $\frac{\bar k\log N}{N^{1-\alpha}}$ in the truncated distance function $h_{\bar k}(\cdot).$
\cite{BerGamTsi_01} has shown that Lyapunov
drift analysis can be used to obtain geometric upper bounds
on the tail probability of discrete-time Markov chain at steady state (Theorem 1 in \cite{BerGamTsi_01}). which can be used to obtain geometric-type upper bounds for continuous-time Markov chains by uniformization of CTMC (see Lemma 4.1 in \cite{WanMagSri_18}). To make the paper self-contained, we state Lemma 4.1 in \cite{WanMagSri_18} next and use it to prove Lemma \ref{ssc:tail}.

\begin{lemma} \label{TailBound}
Let $(X (t): t \geq 0)$ be a continuous-time Markov chain over a countable state space $\mathbb X$. Suppose that it is irreducible, nonexplosive and positive-recurrent, {and $X$ denotes the steady state of $(X (t): t \geq 0).$} Consider a Lyapunov function $V: \mathbb X \to \mathbb{R}^{+}$ and define the drift of $V$ at a state $i \in \mathbb X$ as $$\Delta V(i) = \sum_{i' \in \mathcal X: i' \neq i} q_{i i'} (V(i') - V(i)),$$ where $q_{ii'}$ is the transition rate from $i$ to $i'.$ Suppose that the drift
satisfies the following conditions:

(i) There exists constants $\gamma > 0$ and $B > 0$ such that $\Delta V (i) \leq -\gamma$ for any $i \in \mathbb X$ with $V(i) > B.$

(ii) $\nu_{\max} := \sup\limits_{i,i'\in \mathbb X: q_{i i'} >0} |V(i') - V(i)|< \infty.$

(iii) $\bar q := \sup\limits_{i \in \mathbb X} (-q_{ii}) < \infty.$

Then for any non-negative integer $j$, we have
$$\Pr\left(V(X) > B + 2 \nu_{\max} j\right) \leq \left(\frac{q_{\max}\nu_{\max}}{q_{\max}\nu_{\max} + \gamma}\right)^{j+1},$$ where $$q_{\max}=\sup\limits_{i \in \mathbb X} \sum_{i' \in \mathbb X: V(i) < V(i')} q_{ii'}.$$ \qed
\end{lemma}

For Lyapunov function $V(s)$ in \eqref{ly-func}, it is easy to check 
\begin{align*}
q_{\max}\leq N ~~~\text{and}~~~ v_{\max}\leq \frac{1}{N}.
\end{align*}
We next define 
\begin{align*}
B=\frac{1}{4N^\alpha} ~~~\text{and}~~~ \gamma=\frac{k-1}{b}\frac{\log N}{N^{1-\alpha}}.
\end{align*}
Based on Lemma \ref{TailBound} with $j=\frac{N^{1-\alpha}}{8},$ we have 
\begin{align*}
\Pr\left(V(S)\geq \frac{1}{2N^\alpha}\right) \leq &\left(\frac{1}{1+\frac{k-1}{b}\frac{\log N}{N^{1-\alpha}}}\right)^{\frac{N^{1-\alpha}}{8}}\\
\leq& \left(1 - \frac{k-1}{2b}\frac{\log N}{N^{1-\alpha}}\right)^{\frac{N^{1-\alpha}}{8}}\\ \leq& e^{-\frac{(k-1)\log N}{16b}} = N^{-2r}
\end{align*}
where the first inequality holds because $\frac{N^{1-\alpha}}{32\log N} > r$ implies that $\frac{k-1}{b}\frac{\log N}{N^{1-\alpha}}<1;$ the second inequality holds because $\left(1-\frac{1}{x}\right)^x\leq \frac{1}{e}$ for $x\geq 1;$ and the last inequality holds because $k=32br+1,$ where $\log N$ is needed to establish so that the probability can be made arbitrarily small when $N$ is sufficiently large.  

\section{Proof of Lemma \ref{lemma:ssc-term}}
According to Lemma \ref{ssc:tail}, we  consider SSC term \eqref{G-expansion-SSC} in two regions, $\Omega$ and its complementary $\bar \Omega,$ as follows
\begin{align}
&E\left[N^\alpha\left(\sum_{i=1}^{b} S_i-1 -\frac{\bar k\log N}{N^{1-\alpha}}\right)^r\left(1-S_1\right)\mathbb{I}_{\sum_{i=1}^{b} S_i> \eta+\frac{1}{N}}\right]\nonumber\\
=& E\left[N^\alpha\left(\sum_{i=1}^{b} S_i-1 -\frac{\bar k\log N}{N^{1-\alpha}}\right)^r\left(1-S_1\right)\mathbb{I}_{V(S)\leq \frac{1}{2N^\alpha}}\mathbb{I}_{\sum_{i=1}^{b} S_i> \eta+\frac{1}{N}}\right]\label{SSC-term1}\\
+&E\left[N^\alpha\left(\sum_{i=1}^{b} S_i-1 -\frac{\bar k\log N}{N^{1-\alpha}}\right)^r\left(1-S_1\right)\mathbb{I}_{V(S)>\frac{1}{2N^\alpha}}\mathbb{I}_{\sum_{i=1}^{b} S_i> \eta+\frac{1}{N}}\right].\label{SSC-term2}
\end{align}
The term \eqref{SSC-term1} is related to the case when the system state is in region $\Omega,$ where $V(s) \leq \frac{1}{2N^\alpha}.$ 
Consider $\sum_{i=1}^{b} s_i> \eta+\frac{1}{N}$ (otherwise $\eqref{SSC-term1} = 0$), then we have $V(s) = 1 - s_1.$  In this case, $V(s) \leq \frac{1}{2N^\alpha}$ implies $s_1 \geq 1 - \frac{1}{2N^\alpha}.$ Therefore, we have
$$\eqref{SSC-term1} \leq \frac{1}{2}E\left[\left(\max\left\{\sum_{i=1}^{b} S_i-1 -\frac{\bar k\log N}{\sqrt{N}},0\right\}\right)^r\right].$$
The term \eqref{SSC-term2} consider the case when the system state is in the region $\bar \Omega.$ We apply the tail bound in Lemma \ref{ssc:tail} to get 
$$\eqref{SSC-term2} \leq \frac{b^r}{N^{2r-\alpha}}.$$
Therefore, Lemma \ref{lemma:ssc-term} holds.

\section{Proof of Lemma \ref{lemma:S3}} \label{app:S3}
Define test function to be $f(s) = \sum_{i=3}^b s_i$ in $G f(S).$ We have 
\begin{align}
G f(s) =& \sum_{i=3}^{b}\lambda N (A_{i-1}(s)-A_i(s)) (f(s + e_i) - f(s)) \nonumber\\
&+  N (s_i - s_{i+1}) (f(s - e_i)-f(s)) \nonumber\\
=& \sum_{i=3}^{b}\lambda N (A_{i-1}(s)-A_i(s)) \left(\frac{1}{N}\right) +  N (s_i - s_{i+1}) \left(-\frac{1}{N}\right) \nonumber\\
=& \lambda \left(A_{2}(s)-A_b(s) \right) -s_3. \nonumber
\end{align} 

According to the steady-state condition $E[ G f(S) ] = 0,$ we have
$$E[S_3] = \lambda E[A_2(S) - A_b(S)].$$

\end{document}